\documentclass[12pt]{amsart}

\usepackage[dvips]{graphicx}
\usepackage{amssymb}
\usepackage{amsmath}
\usepackage{amsthm}
\usepackage{amsfonts}
\usepackage{amsthm,amscd}
\usepackage{amsbsy}
\usepackage{bm}
\usepackage{url} 
\usepackage{setspace}
\usepackage{float}
\usepackage{psfrag}
\usepackage{cite}
\usepackage[pagebackref,colorlinks, linkcolor=blue]{hyperref}
\usepackage{calc,amssymb,amsthm,amsmath,amscd, eucal}
\usepackage[left=1in, top=1in,right=1in, bottom=1in]{geometry}
\usepackage{tikz}
\usepackage{color}
\usepackage{tikz-cd}
\usepackage{pgfplots}
\tikzset{mynode/.style={inner sep=2pt,fill,outer sep=0,circle}}
\usetikzlibrary{matrix, calc, arrows}
\usepackage{wrapfig}
\usepackage[all,cmtip]{xy}

\newcommand{\CM}{Cohen-Macaulay}

\newcommand{\m}{\mathfrak{m} }
\newcommand{\fm}{\mathfrak{m} }

\newcommand{\la}{\lambda}

\newcommand{\R}{\mathcal{R}}

\newcommand{\TT}{\mathcal{T} }

\newcommand{\Z}{\mathbb{Z} }

\newcommand{\NN}{\mathbb N}

\newcommand{\pd}{\operatorname{pd}}

\newcommand{\depth}{\operatorname{depth}}

\newcommand{\Soc}{\operatorname{Soc}}

\newcommand{\gr}{\operatorname{gr}}

\newcommand{\hgt}{\operatorname{height}}

\newcommand{\Tor}{\operatorname{Tor}}

\theoremstyle{plain}

\newtheorem{theorem}{Theorem}[section]
\newtheorem{corollary}[theorem]{Corollary}
\newtheorem{lemma}[theorem]{Lemma}
\newtheorem{proposition}[theorem]{Proposition}

\theoremstyle{definition}

\newtheorem{remark}[theorem]{Remark}
\newtheorem{example}[theorem]{Example}

\newtheorem*{question*}{\it Question}
\newtheorem*{example*}{Example}
\newtheorem*{claim*}{\it Claim}
\newtheorem*{note*}{\it Note}

\restylefloat{figure}

\def\be{\begin{equation}}
\def\ee{\end{equation}}
\def\ben{\begin{eqnarray}}
\def\een{\end{eqnarray}}

\title[]{Associated graded rings
of the   filtration of    \\ [2mm]  tight closure of powers of parameter ideals }
\author{Saipriya Dubey}
\address{Department of Mathematics, Indian Institute of Technology Bombay
Powai, Mumbai, India - 400076}
\email{sdubey@math.iitb.ac.in}
\author{Jugal Verma}
\email {jkv@iitb.ac.in}

\begin{document}

\maketitle
\begin{abstract}
Let $I$ be an ideal generated by a system of parameters in an excellent Cohen-Macaulay local domain. We show  that the associated graded ring $G^*(I)$ of the filtration $\{(I^n)^*: n\in \NN\}$ is  Cohen-Macaulay. We prove that if $R$ is an excellent  Buchsbaum local domain then  $G^*(I)$ is a Buchsbaum module over the Rees ring $\mathcal R^*(I)=\oplus_{n\in \NN}(I^n)^*.$ We provide quick proofs of well-known results of I. Aberbach, Huneke-Itoh and Huneke-Hochster about the filtration  $\{(I^n)^*: n\in \NN\}$ in excellent local domains. An important  tool  used in the proofs is a deep result due to M. Hochster and C. Huneke which states that the   absolute integral closure of an excellent local domain is a big Cohen-Macaulay algebra \cite{HocHun1992}. We compute the tight closure of $I^n$ where $I$ is generated by homogeneous system of parameters having the same degree $e$  in the  hypersurface ring  $R=\mathbb{F}_p[X_0,\ldots ,X_d]/(X_0^r+\cdots+X_d^r).$ In such cases we prove that $G^*(I)$ is Cohen-Macaulay. We provide conditions on $r, d, e$ for the Rees algebra $\mathcal R^*(I)$ to be Cohen-Macaulay.
\end{abstract}
\thispagestyle{empty}
\section{\bf Introduction}
Throughout this article, $R$ denotes a commutative Noetherian local ring with maximal ideal $\m.$  Let $J^*$ denote the tight closure of an ideal $J.$ In this paper, we study some properties of the tight closure of powers of an ideal.  
The depth of  the associated graded ring 
\[G^*(I)=\bigoplus_{n\geq 0} (I^n)^*/(I^{n+1})^*\] of the filtration $\TT=\{(I^n)^*\}$ where $I$ is an $\m$-primary ideal is an important invariant for understanding the tight Hilbert polynomial   $P^*_I(x)$ for the tight Hilbert function $H^*_I(n)=\ell(R/(I^n)^*).$ By a theorem of D. Rees, 
\cite[Theorem 1.4]{reesau}, if  $R$ is analytically unramified Noetherian 
local ring, then the Rees algebra $\R^*(I)=\bigoplus_{n\geq 0}(I^n)^*t^n$ is a finite module over the Rees algebra $\R(I)=\bigoplus_{n\geq 0}I^nt^n.$  This means that $G^*(I)$ is a finitely generated graded module over the associated graded ring $G(I)=\bigoplus_{n\geq 0}I^n/I^{n+1}.$
Therefore the  \textbf{tight Hilbert function} defined by $H_I^*(n)=\lambda(R/(I^n)^*)$ coincides with a  polynomial, $P_I^*(n)$, of degree $d=\dim R$ for large $n$ and is known as the \textbf{tight Hilbert polynomial of $I$}. We write
$$P_I^*(n)=e_0^*(I)\binom{n+d-1}{d}-e_1^*(I)\binom{n+d-2}{d-1}+\cdots+(-1)^de_d^*(I)$$ where $e_i^*(I)\in \mathbb{Z}.$
Here $e^*_0(I)=e(I),$ the multiplicity of $I$,  $e^*_1(I)$ is known as the \textbf{tight Chern number of $I$} and  $e^*_i(I)$ for $i=0,1,\ldots,d$ are called the \textbf{tight Hilbert coefficients of $I.$}

The vanishing of the Chern number of $I$ gives information about the local ring $R.$  Recall that a Noetherian local ring of prime characteristic is called an $F$-rational ring if all ideals generated by parameters are tightly closed. It is proved  in \cite{GMV}, that in  a \CM~  analytically unramified  local ring of  prime characteristic,  the tight Chern number of some ideal generated by a system of parameters is zero if and only if $R$ is $F$-rational. 

Let $I$ be an ideal generated by a system of parameters and $R$ be analytically unramified Cohen-Macaulay  local ring.
If the depth of $G^*(I)\geq d-1$ where $d=\dim R,$ then by a result of 
S. Huckaba and T. Marley, \cite{HucMar} the tight Hilbert coefficients of $I$ for $j=1, 2, \ldots, d$ are given by
\[e_j^*(I)=\sum_{n\geq j}\binom{n-1}{j-1}\ell((I^n)^*/I(I^{n-1})^*).\]
Therefore, it is useful to know the depth of $G^*(I).$ 
If $x_1, x_2, \ldots, x_h $ are elements of a commutative Noetherian ring and
$\text{ht}(x_1, x_2, \ldots, x_h)=h$ then we say that $x_1, \ldots, x_h$ are {\bf parameters}  and the ideal $(x_1, x_2, \ldots, x_h)$ is called a {\bf parameter ideal.}
We prove the following result about depth of $G^*(I)$ in  Section \ref{sec: G*(I)} .
\begin{theorem}
Let $(R,\m)$ be an excellent Cohen-Macaulay local domain. Let $I$ be a parameter ideal.  Then $\depth G^*(I)\geq \text{ht } I.$ In particular if $I$ is generated by a system of parameters, then  $G^*(I)$ is  Cohen-Macaulay.
\end{theorem}

 Recall that the plus closure $R^+$ or absolute integral closure of a Noetherian domain $R$ is the integral closure of $R$ in an algebraic closure of the quotient field of $R.$
One of the main  tools used  is the following result  due to Ian Aberbach.
\begin{theorem}\label{main}
 Let $I$ be a parameter ideal in an excellent local domain $R$. Then for all $n\geq 1,$
\begin{eqnarray*}
(I^n)^*=I^nR^+\cap R.
\end{eqnarray*} 
\end{theorem}

We present a short proof of Theorem \ref{main} due to Karen Smith in Section \ref{sec: genprop}. Several consequences of Theorem \ref{main} are proved. Let $\overline{I}$ denote the integral closure of an ideal $I.$ A celebrated result due to 
C. Huneke and  S. Itoh states that if $R$ is a Noetherian ring and $I$ is generated by a regular sequence  then  for all $n\geq 1,$
\[I^{n+1}\cap \overline{(I^n)}=\overline{I}I^n.\]
As a consequence of Theorem \ref{main}, we present a short proof of the following 
 result of Hochster and Huneke \cite[Theorem 8.2]{HH1994}.
 
 \begin{theorem}
  Let $R$ be an  excellent local domain and $I$ be a  parameter ideal. Then for all $n\geq 1,$
\[I^{n+1}\cap (I^n)^*=I^*I^n.\]
\end{theorem}

In Section \ref{sec: G*(I)}, we calculate tight closure of powers of an  ideal generated by homogeneous system of parameters of same degree in hypersurface ring $R=\mathbb{F}_p[X_0,\ldots ,X_d]/(X_0^r+\cdots +X_d^r).$ We use these explicit formulas to show that $G^*(I)$ is {\CM} and find conditions under which $\mathcal R^*(I)$ is Cohen-Macaulay.
\begin{proposition} Let $R=\frac{\mathbb{F}_p[X_0,X_1,\ldots,X_{d}]}{(X_0^r+X_1^r+\cdots+X_d^r)},$ where $p\nmid r; r,d\geq2$ and $p>(d-1)r-d.$ Let $f_1,\ldots,f_d$ be homogeneous system of parameters in $R$ such that $\deg (f_i)=e$ for $1\leq i \leq d,$ $J=(f_1,\ldots,f_d)$ and $\fm =(x_0,x_1,\ldots ,x_d).$ Then\\
{\rm (a)} $(J^n)^*=J^n+\fm^{(n-1)e+de},$\\
{\rm (b)} $G^*(J)$ is Cohen Macaulay,\\
{\rm (c)} $r_J(\fm^e)=\left[ \frac{r-1-d}{e}\right]+d,$ $r^*(J)\le \left[ \frac{r-1-d}{e}\right]+1.$\\
{\rm (d)} If $\left[\frac{r-1-d}{e}\right] \leq d-2$ then $\mathcal R^*(J)$ is Cohen-Macaulay.
\end{proposition}

\noindent
{\bf Acknowledgements.} We thank Karen Smith for discussions and for explaining us a simple proof of Theorem \ref{main}.
 
\section{\bf Some properties of the tight closure of powers of parameter ideals}\label{sec: genprop}
In this section, we assume that $R$ is a Noetherian local ring of prime characteristic $p.$ For a ring $R,$ the subset of $R$ consisting of all the elements which are not contained in any minimal prime ideals of $R,$ is denoted by $R^0.$   Let $q=p^e.$ If $I=(a_1,\ldots,a_r)$ is an ideal of $R$ then $I^{[q]}=(a_1^q,\ldots,a_r^q).$ For any positive integer $n,$ $(I^n)^{[q]}=(I^{[q]})^n.$ An element $x\in R$ is said to be in the \textbf{tight closure $I^*$} of $I$ if there exists $c\in R^0$ such that $cx^q\in I^{[q]},$ for all sufficiently large $q.$
An ideal $I$ is called tightly closed if $I=I^*.$  The tight filtration of $I$ is given by $\mathcal{T}=\lbrace(I^n)^*\rbrace_{n\in \mathbb{Z}}.$ We note that $(I^n)^*\subset \overline{I^n},$ for all $n\geq 1.$

We begin by recalling a few  results about big Cohen-Macaulay algebras and modules needed later.
Let $R$ be an integral domain with fraction field $K.$ Let $\overline{K}$ be  an  algebraic closure of $K.$ The integral closure of $R$ in $\overline{K},$ denoted $R^+,$ is called the {\bf absolute integral closure of $R.$}  Let $(R,\m)$ be a local Noetherian ring and $M$ be a non-zero $R$-module, we say that $M$ is a \textbf{big Cohen-Macaulay} $R$-module if some system of parameters for $R$ is a regular sequence on $M.$  We say that
$M$ is a \textbf{balanced big Cohen-Macaulay} $R$-module if every system of parameters for $R$ is a regular sequence on $M.$
  We need the following  results for the proof of the Main Theorem. 
\begin{theorem}{{\bf Colon capturing, \cite[Theorem 10.1.9]{BH}}}
\label{col_cap1}
Let $R$ be an equidimensional residue class ring of a Cohen-Macaulay local ring $A$, and $x_1,x_2,\ldots, x_t$ be  parameters of $R.$ Then for
$j=0,1,\ldots,t-1.$
$$(x_1,\ldots,x_j):_Rx_{j+1}\subset (x_1,\ldots,x_j)^*.$$
\end{theorem}
A very useful version of colon capturing is described in the next result.

\begin{theorem}
{\normalfont\cite[Theorem 2.2]{ahs}}\label{ahs}
Let $(R,\m)$ be an equidimensional excellent local ring.  Let
 $A=\Z/p\Z[x_1, x_2, \ldots, x_d]$ be a subring of $R$ and  $x_1, x_2, \ldots, x_d$ be a system of parameters for $R.$ Let $I$ and $J$ be ideals of $A$ generated by monomials in $x_1, x_2, \ldots, x_d.$ Then 
\begin{eqnarray*}
(IR)^*:_RJR)&=& ((I:_AJ)R)^*\\
(IR)^*\cap (JR)^*&=& ((I\cap J)R)^*
\end{eqnarray*}

\end{theorem}
\begin{theorem} {\normalfont\cite[Theorem 11.1]{Hochster1994a}}\label{Hoc_94}
Let $(R,\m)$ be a complete local domain of prime characteristic. Then for any $R$-ideal $I$, $u\in I^*$
if and only if $u\in IB\cap R,$ where $B$ is some big Cohen–Macaulay $R$-algebra.
\end{theorem}
\begin{theorem} {\normalfont\cite[Theorem 1.1]{HocHun1992}}\label{HocHun92bigCM}
Let $(R,\m)$ is an excellent local domain of prime characteristic $p.$ Then $R^+$ is a big Cohen-Macaulay $R$-algebra.
\end{theorem}

\begin{theorem}{\normalfont\cite[Theorem 5.1]{ksmith}}\label{ksmith}
Let $R$ be a Noetherian locally excellent domain of prime characteristic $p.$ If $I$ is any parameter ideal of $R$ then $IR^+\cap R=I^*.$ 
\end{theorem}

The above theorem holds for powers of $I.$ It follows from \cite[Theorem 3.1]{Aber_94}. However, we present a proof which has been communicated to us by Professor Karen Smith. 

\begin{theorem}\label{tytcl_frm_plus}
Let $(R, m)$ be an excellent local domain of prime characteristic $p.$ Let $I= (x_1,\ldots, x_t)$ be any  parameter ideal of  $R.$ Then
$$(I^n)^*=I^nR^+\cap R.$$
\end{theorem}

\begin{proof} ({\bf Karen Smith.})
Use induction on $n.$ The  $n=1$ case  is the  Theorem \ref{ksmith}.
Let $n\geq 2.$ By  Theorems  \ref{Hoc_94} and \ref{HocHun92bigCM}, it is enough to show that$$(I^n)^* \subset I^nR^+\cap R.$$ We assume that the result is true for $n-1.$ Let $z\in (I^n)^*.$
Since $(I^n)^* \subset (I^{n-1})^*,$ by induction hypothesis, $z = \sum_{i=1}^m a_i x^{\mu_i}$ where the coefficients $a_i \in R^+$ and $x^{\mu_i}$ ranges through all degree $n-1$ monomials in $x_1,x_2,\ldots,x_t.$
Let $1\leq i\leq m.$ Then 
$$a_i x^{\mu_i}\in ((x_1,x_2,\ldots,x_t)^n+J_i)^*$$ in $R^+$ where $J_i$ is the ideal generated by all the degree $n-1$ monomials in $x_1,x_2,\ldots,x_t$ except $x^{\mu_i}.$ Since $x^{\mu_i}\notin I^n+J_i,$ and $a_i\in ((x_1,x_2,\ldots,x_t)^n+J_i)^* :x^{\mu_i},$ by Theorem \ref{ahs} it follows that $a_i\in I^*=IR^+.$ Therefore $z\in I^nR^+\cap R.$
\end{proof}
There are several interesting consequences of  Theorem \ref{tytcl_frm_plus}.
These may be proved without using it. However,  Theorem \ref{tytcl_frm_plus},
enables us to provide short  unified proofs.

Huneke \cite[Theorem 4.7]{huneke1987} and Itoh\cite[Theorem 1]{itoh1988} independently proved that in a $d$-dimensional local ring, for an ideal generated by a regular sequence, for all $n\geq 1$ 
\[I^n\cap \overline{I^{n+1}}=I^n\overline{I}.\] The analogue of Huneke-Itoh intersection theorem for tight closure of an ideal generated by regular sequence has been proved in \cite[Theorem 8.20] {HH1994}. However, we prove it in an excellent local Noetherian domain where $I$ is a parameter ideal. This result plays an important role in the study of the tight Hilbert polynomial.

\begin{corollary}\label{extyt}
Let $(R,\m)$ be an excellent local Noetherian domain of prime characteristic.  Let $I$ be an ideal generated  by parameters in $R.$ Then for all $n\geq 0,$
	$$I^n\cap (I^{n+1})^*=I^n I^*.$$ 
\end{corollary} 
\begin{proof}
	Let $I$ be the  ideal generated by $x_1,x_2,\dots,x_t.$
	The inclusion $I^n I^*\subset I^n\cap (I^{n+1})^* $ is clear since $I^*J^*\subset (I^*J^*)^*=(IJ)^*.$ For the reverse inclusion, let $a\in I^n\cap (I^{n+1})^* =I^n \cap I^{n+1}R^+$ where the equality follows from Theorem \ref{tytcl_frm_plus}. 
	We write $$a=\sum_{\vert \alpha \vert =n}y_{\alpha}x^{\alpha}=\sum_{\beta \in \mathbb{N}^d ,\;\; \vert \beta\vert =n+1}z_{\beta}x^{\beta}$$
	for some $y_{\alpha}\in R$ and $z_{\beta}\in R^+.$ As $x_1,x_2,...,x_t$ forms a regular sequence in $R^+$ we have $y_{\alpha}\in IR^+.$ Hence  $y_{\alpha}\in IR^+\cap R=I^*.$ This implies that $a\in I^nI^*.$ 
\end{proof}
The next result is due to K.-i. Watanabe for regular sequences in local rings of prime characteristic \cite[Corollary 4.3]{GMV}. We generalise it for parameter ideals in  excellent local domains of prime characteristic.
\begin{corollary}
Let $(R,\m)$ be a $d$-dimensional excellent local Noetherian domain of prime characteristic.  Let $I$ be an ideal generated by parameters in $R.$ If $I$ is tightly closed then $I^n$ is tightly closed for all $n\geq 1.$
\end{corollary}
\begin{proof}
Use induction on $n.$ For $n=1$ the result is true. Suppose that $I,I^2,\ldots,I^{n-1}$ are tightly closed. Let $x\in (I^n)^*.$ Since $(I^n)^*\subset (I^{n-1})^*=I^{n-1},$ using Corollary \ref{extyt} it follows that $x\in (I^n)^*\cap I^{n-1}=I^*I^{n-1}=I^n.$ Hence $I^n$ is also tightly closed. 
\end{proof}

\section{\bf Cohen-Macaulay and Buchsbaum properties of $G^*(I)$}\label{sec: G*(I)}
In this section we discuss the Cohen-Macaulay property of the associated graded ring $G^*(I)$ for an ideal $I$ generated by a system of parameters in an excellent Cohen-Macaulay local domain. We also consider the Buchsbaum property of $G^*(I)$ if $R$ is an excellent Buchsbaum local domain.
\begin{theorem}\label{thm: main}
Let $(R,\m)$ be an excellent  \CM~ local domain of prime  characteristic  and $I=(x_1, x_2, \ldots, x_h)$ be a parameter ideal. Then $\depth G^*(I)\geq h.$ In particular, if $I$ is generated by a system of parameters then $G^*(I)$ is Cohen-Macaulay.
\end{theorem}

\begin{proof}
By Valabrega-Valla  \cite{vv} it is enough to show that for all $n\geq0,$ the $VV_n$ property
\[I\cap (I^{n+1})^*=I(I^n)^*\]
is satisfied. 
We prove $VV_n$ by  induction on $n.$ The statement is clearly true for $n=0.$ Suppose that $n\geq 1$ and 
${VV}_{n-1}$ is true. We prove $VV_n.$ Let 
$$z=\sum_{i=1}^ha_ix_i\in (I^{n+1})^*=I^{n+1}R^+\cap R.$$ We show that $a_i\in (I^n)^*$ for all $i=1, 2, \ldots, h.$  Since  $z\in I \cap (I^{n})^*,$ by induction hypothesis, $a_i\in (I^{n-1})^*\subset I^{n-1}R^+$ for all $i=1, 2, \ldots, h.$ We can write $$a_i=
\sum_{\alpha \in \NN^{h}, \;\; |\alpha|=n-1}a_{i\alpha}x^{\alpha} \text{  for some }a_{i\alpha}\in R^+.$$
 Since
$z=\sum_{i=1}^h a_ix_i \in (I^{n+1})R^+$ and $x_1, x_2, \ldots, x_h$ is an $R^+$-regular sequence, $a_{i\alpha}\in IR^+$ for all $i$ and all $\alpha.$ Hence $a_i\in (I^n)R^+\cap R=(I^n)^*$ for all $i=1, 2, \ldots, h.$ This proves $VV_n$ for all $n.$ Hence $\depth G^*(I)\geq h.$ In particular if $h=d$ then $G^*(I)$ is Cohen-Macaulay.

\end{proof}

\begin{theorem}
Let $(R,\fm)$ be an excellent local Buchsbaum domain of prime characteristic $p.$ Let $I=(a_1,a_2,\cdots, a_d)$ be an $\fm$-primary parameter ideal. Then $G^*(I)$ is Buchsbaum as an $\mathcal{R}^*(I)$-module.
\end{theorem}
\begin{proof}
Using  \cite[Theorem 1.2]{sk}, it is enough to prove that for all $n\ge 3,$
\begin{align*}
(a_1^2, \dots, a_d^2)\cap (I^n)^*=(a_1^2, \dots, a_d^2)(I^{n-2})^*.
\end{align*}
Consider the following statement:
\begin{equation} \label{eq:Bn}
\text{Let } z= \sum_{i=1}^d b_ia_i^2\in (I^{n})^* \text{ for some }  b_i\in R \text{ then } b_i\in (I^{n-2})^*.
\end{equation}
We  prove \eqref{eq:Bn} for $n\ge 3$ using induction on $n.$ Let $n=3$ and $z=\sum_{i=1}^d b_ia_i^2 \in (I^3)^*\subset I^3R^+.$ Let $F(X_1,\dots, X_d)= \sum_{i=1}^d b_iX_i^2,$ then $F(a_1,\dots, a_d) \in I^3R^+.$ Since, $a_1,\dots, a_d$ is an $R^+$-regular sequence, it follows that $b_1, \dots, b_d \in I R^+.$ Hence $b_1,\ldots,b_d\in I^*$ which proves \eqref{eq:Bn} for $n=3.$
\noindent Now, we prove that \eqref{eq:Bn} holds for $n+1$ assuming that \eqref{eq:Bn} is true for $n.$
Let $ z= \sum_{i=1}^d b_ia_i^2\in (I^{n+1})^*\subset (I^n)^*.$ By induction hypothesis it follows that $b_i \in (I^{n-2})^* \subset I^{n-2}R^+.$ Therefore, 
\begin{align*}
b_i=\sum_{|\alpha|=n-2} c_\alpha a^\alpha, \text{ where }a^\alpha=a_1^{\alpha_1}\cdots a_d^{\alpha_d}.
\end{align*}
Note that 
$$z= \sum_{i=1}^d \left( \sum_{|\alpha|=n-2} c_\alpha a^\alpha \right)a_i^2 =  \sum_{i=1}^d \sum_{|\alpha|=n-2} c_\alpha a^{\alpha+2e_i}\in I^{n+1}R^+,$$
where $e_i=(0,0,\cdots, 1,0,\cdots,0)$ with $1$ as the $i^{th}$  coordinate. Therefore  $c_\alpha\in IR^+$ and hence 
$b_i\in I^{n-1}R^+\cap R=(I^{n-1})^*.$ 
\end{proof}

 We have computed the tight closure of $I^n,$ where $I=(x_1,\ldots,x_d)$ in a local ring of the type $R = \mathbb{F}_p[[X_1, X_2, \ldots , X_n]]/(f),$ where $f$ is an irreducible polynomial in \cite{dmv} and \cite{GMV2019}. In several cases $(I^n)^* = I^n + \m^{n+r}$ for all $n \geq 1$ and for a fixed integer $r.$ Here $\m$ denotes the maximal homogeneous ideal of $R.$  In this section we prove a stronger version where we consider $I$ to be ideal generated by homogeneous system of parameters having the same degree. Using these formulas, we are able to conclude that $G^*(I)$ is {\CM}. The connection between the back twists in a minimal resolution and the socle degree of an artinian ring plays a crucial role in the proof. 
The observation given below is the content of \cite[Lemma 1.3]{KU92}.
\begin{lemma}\label{Soc_twist_obs}
Let $S=k[X_1,\ldots,X_d]$ be a positively graded polynomial over a field $k$ where $\deg(X_i)=e$ for $1\leq i \leq d,$ $S/J$ be an Artinian graded ring and $\mathbb{F}$ be a minimal graded resolution of $S/J$ by free $S$-modules given by
\begin{align*}
\mathbb{F}: 0 \longrightarrow \mathbb{F}_d=\oplus_{i=1}^rS(-b_i) \longrightarrow \mathbb{F}_{d-1}\longrightarrow \cdots \rightarrow \mathbb{F}_0 \longrightarrow 0.
\end{align*}
Then there exists a graded isomorphism 
$$\Soc(S/J)\cong \oplus_{i=1}^rk(-b_i+de).$$ 

\end{lemma} 
\begin{proof} Consider the graded $S$-module $\Tor_d^S(S/J,k).$ The $d^{th}$ homology of the complex obtained by tensoring $\mathbb{F}$ with $k$ yields $\Tor_d^S(S/J,k)=H_d(\mathbb{F}\otimes k)=\oplus_{i=1}^rk(-b_i).$ Another way to compute $\Tor_d^S(S/J,k)$ is via the Koszul compex on $X_1,\ldots ,X_d,$ $\mathbb{K}$ say.
\begin{align*}
\mathbb{K}: 0 \longrightarrow S(-de) \xrightarrow{\begin{bmatrix}
    X_{1} & X_{2} & \cdots & X_{d}  
 \end{bmatrix}} S^d(-(d-1)e)\longrightarrow \cdots \rightarrow S^d(-e) \longrightarrow S \longrightarrow S/\m,
\end{align*}
Tensoring by $S/J$ we obtain 
\begin{align*}
\mathbb{K}: 0 \longrightarrow \ker \phi \longrightarrow (S/J)(-de) \xrightarrow{\phi} (S/J)^d(-(d-1)e)\longrightarrow \cdots \rightarrow (S/J)^d(-e) \longrightarrow (S/J) \longrightarrow S/\m,
\end{align*}
where $\phi$ is the multipication map by  $\begin{bmatrix}
    X_{1} & X_{2} & \cdots & X_{d}  
 \end{bmatrix}.$
 Since $\ker \phi=\frac{J:(X_1,\ldots ,X_d)}{J}=\Soc(S/J),$ it follows that 
 $$\Tor^S_d(S/J,k)=H_d(\mathbb{K}\otimes S/J)=\Soc(S/J)(-de).$$
Hence $\Soc(S/J)\cong \oplus_{i=1}^rk(-b_i+de).$
\end{proof}

\begin{lemma} \label{Soc_twist_lemma}
Let $R$ be a finitely generated $k$-algebra which is Cohen-Macaulay graded ring with $\dim(R)=d$ over a field $k.$ Let $f_1,\ldots,f_d$ be homogeneous system of parameters of $R$ such that $\deg(f_i)=e$ for all $1\leq i \leq d$ and $I=(f_1,\ldots ,f_d).$ Then there exists a one-one correspondence between the socle degrees of $S/\m^n$ and the last twist in a projective resolution of $R/I^n,$ where $S=k[f_1,\ldots ,f_d]$ is a polynomial ring and $\m=(f_1,\ldots,f_d)S.$
\end{lemma}

\begin{proof}
Since  $f_1,\ldots ,f_d$ is a homogeneous system of parameters, $R$ is a finite $S$-module. By Auslander-Buchsbaum formula it follows that $\pd _SR+\depth_SR=d.$ Since $\depth_SR=d,$ $R$ is a free $S$-module. Let a graded free resolution of $S/\m^n$ be given by
\begin{align*}
 0 \longrightarrow \oplus_i S(-b_{d_i}) \longrightarrow \cdots \longrightarrow \oplus_iS(-b_{1_i}) \longrightarrow S \longrightarrow S/\m^n \longrightarrow 0.
\end{align*}
Since $R$ is a flat $S$-module, tensoring the above gives free resolution for $R/I^n$ as $S/\m^n\otimes R\cong R/I^n,$ where the twists remain unchanged as the maps are induced. 
\begin{align*}
 0 \longrightarrow \oplus_i R(-b_{d_i}) \longrightarrow \cdots \longrightarrow \oplus_iR(-b_{1_i}) \longrightarrow R \longrightarrow R/I^n \longrightarrow 0.
\end{align*}
Using Lemma  \ref{Soc_twist_obs} we conclude that $$\Soc(S/\m^n)\cong \oplus_ik(-b_{d_i}+de)$$
which implies the result by comparing degrees on either side of the graded isomorphism.
\end{proof}
Next result gives a formula for $(I^n)^*$ when $I$ is  the ideal generated by homogeneous system of parameters of the same degree. 

\begin{proposition} \label{tytclhomsop} Let $R=\frac{\mathbb{F}_p[X_0,X_1,\ldots,X_{d}]}{(X_0^r+X_1^r+\cdots+X_d^r)},$ where $p\nmid r; r,d\geq2$ and $p>(d-1)r-d.$ Let $f_1,\ldots,f_d$ be homogeneous system of parameters in $R$ such that $\deg (f_i)=e$ for $1\leq i \leq d,$ $J=(f_1,\ldots,f_d)$ and $\fm =(x_0,x_1,\ldots ,x_d).$ Then\\
{\rm (a)} $(J^n)^*=J^n+\fm^{(n-1)e+de},$\\
{\rm (b)} $G^*(J)$ is Cohen Macaulay,\\
{\rm (c)} $r_J(\fm^e)=\left[ \frac{r-1-d}{e}\right]+d,$ $r^*(J)\le \left[ \frac{r-1-d}{e}\right]+1.$\\
{\rm (d)} If $\left[\frac{r-1-d}{e}\right] \leq d-2$ then $\mathcal R^*(J)$ is Cohen-Macaulay.
\end{proposition}
\begin{proof}
(a) We begin by observing that $R$ satisfies the Strong Vanishing Conjecture by \cite[Theorem 6.1]{Huneke1998}. Put $S=\mathbb{F}_p[f_1,\ldots,f_d]$ and $\mathfrak{n}=(f_1,\ldots,f_d)S.$ Then $\Soc(S/\mathfrak{n}^n)=\frac{\mathfrak{n}^n:\mathfrak{n}}{\mathfrak{n}^n}=\frac{\mathfrak{n}^{n-1}}{\mathfrak{n}^n}.$ Since the Socle of $S/\mathfrak{n}^n$ is generated by all the monomials of degree $n-1$ in $f_1,\ldots,f_d,$ it implies that degree of any minimal generator of the socle is  $(n-1)e.$ Using Lemma \ref{Soc_twist_lemma} for the pair $S/\mathfrak{n}^n$ and $R/J^n$ we get a one-to-one correspondence between the set degrees of minimal generators of $\Soc(S/\mathfrak{n}^n)$ and the set of twists $b_{d_i}$ for $i=1,\ldots,r,$ where $b_{d_i}$ refers to the last twist occuring in the projective resolution of $R/J^n.$ Since $\deg(\Soc(S/\mathfrak{n}^n))=(n-1)e,$ by Lemma \ref{Soc_twist_lemma} it implies that $b_{d_1}=\cdots=b_{d_r}$ and $b_{d_i}-de=(n-1)e$ for $1\leq i \leq r.$ Therefore using \cite[Theorem 5.11]{CHKES97}, it follows that $(J^n)^*=J^n+\m^{(n-1)e+de}.$\\
\noindent (b) Since, $G(\fm)\cong R,$ $G(\fm)$ is Cohen Macaulay, using Valabrega-Valla equation, we get 
\begin{align*}
J\cap \fm^{n}=J\fm^{n-e} \text{ for all }n \ge 1.
\end{align*}
Consequently, 
\begin{align*}
J\cap (J^n)^*= J^n+J\cap \fm^{(n-1)e+de}& = J^n+J\fm^{e(n-1)+(d-1)e}\\
& =J\left( J^{n-1}+\fm^{e(n-1)+(d-1)e} \right) = J(J^{n-1})^*.
\end{align*}
Hence $G^*(J)$ is Cohen Macaulay.\\
\noindent (c) Using \cite[Corollary 2.21]{marthesis}, we get $r_J(\fm^e)=\left[\frac{r-1-d}{e} \right]+d.$ Hence $J(\fm^e)^n=(\fm^e)^{n+1}$ for all $n\ge \left[ \frac{r-1-d}{e} \right]+d.$ For $n\ge \left[ \frac{r-1-d}{e}\right]+1,$
\begin{align*}
J(J^n)^* & = J\left(J^n+\fm^{(n-1)e+de}\right)\\
  & = J^{n+1} +J\fm^{(n-1)e+de} \\
  & = J^{n+1}+\fm^{e[n+d]} = (J^{n+1})^*.
\end{align*}
Therefore, $r^*(J)\le \left[ \frac{r-1-d}{e} \right]+1.$\\
$(d)$ Since $R$ and $G^*(I)$ are Cohen-Macaulay, by a result of Viet
\cite[Theorem 2.3]{viet}  $\mathcal R^*(J)$ is Cohen-Macaulay  if and only if $r^*(J)\leq d-1.$
Thus if $ \left[ \frac{r-1-d}{e} \right]+1\leq d-1$ then $r^*(J)\leq d-1.$ Therefore $\mathcal R^*(J)$ is Cohen-Macaulay.
\end{proof}


\begin{thebibliography}{9999}	 

\bibitem {Aber_94} I. M. Aberbach, {\em Tight closure in $F$-rational rings,} Journal Nagoya Math, {\bf 135} (1994), 43-54.

\bibitem {ahs} I. Aberbach, C. Huneke and K. Smith,
{\em Arithmetic Cohen-Macaulayfications using ideals of dimension one,}
Illinois J. Math. {\bf 40} (1996), 518-526.



\bibitem{KU92} K. Andrew and B. Ulrich, {\em If the socle fits,} J. Algebra, {\bf 147} (1992), 63-80.



\bibitem {BH} W. Bruns and J. Herzog, {\em Cohen-{M}acaulay rings,} Cambridge University Press, Cambridge (1993).

\bibitem{dmv} S. Dubey, V. Mukundan and J.K. Verma,
{\em Tight closure of powers of parameter ideals in hypersurface rings and their tight Hilbert polynomials,}{\bf 8} (2021).
		
		
\bibitem {GMV} K. Goel, V. Mukundan, and J. K. Verma, {\em On the vanishing of the normal Hilbert coefficients of ideals,} J. Ramanujan Math. Soc. {\bf 35} (2020), 121-138.

\bibitem{GMV2019} K. Goel, V. Mukundan, and J.K. Verma, \emph{Tight Closure of powers of ideals and
tight Hilbert Polynomials}, Math. Proc. Cambridge Philos. Society {\bf 169} (2020), 335-355.

\bibitem{Hochster1994a} M. Hochster, \emph{Solid closure}, Commutative algebra: syzygies, multiplicities, and birational algebra ({S}outh {H}adley, {MA}, 1992), Contemp. Math. {\bf 159} (1994), 103-172,


\bibitem {HocHun1992} M. Hochster, C. Huneke, {\em Infinite integral extensions and big Cohen-Macaulay
algebras,} Ann. Math. {\bf 135} (1992), 53-89.  

\bibitem {HH1994} M. Hochster and C. Huneke,
{\em Tight closure of parameter ideals and splitting in module finite extensions},
J. Algebraic Geometry {\bf 3} (1994), 599-670.

\bibitem{CHKES97} C. Huneke and Karen E. Smith, {\em Tight closure and the Kodaira vanishing theorem,} J. Reine Angew. Math., 484:127-152, 1997. 2, 6.

\bibitem{HucMar}
S. Huckaba and T. Marley, {\em Hilbert Coefficients and the Depths of Associated Graded Rings}	
J. Lond. Math. Soc., {\bf 56} (1997), 64-76.


\bibitem {huneke1987} C. Huneke, {\em Hilbert functions and symbolic powers}, Michigan Math. J. {\bf 34} (1987), 293-318.


\bibitem {Huneke1998} C. Huneke, {\em Tight closure, parameter ideals, and geometry,} Six lectures on commutative algebra Bellaterra, (1996)  
 edited by J. Elias, Progr. Math. 166, Birkh\"auser, Basel, (1998), 187-239.


\bibitem {itoh1988} S. Itoh,  {\em Integral closures of ideals generated by regular sequences,} J. Algebra {\bf 117} (1988), 390-401.

\bibitem {marthesis} T. Marley, {\em Hilbert functions of ideals in Cohen-Macaulay local rings,} Ph.D. Dissertation, Purdue University (1989). 


\bibitem {reesau} D. Rees, {\em A Note on Analytically unramified Local Rings,} Journal London Math Soc., {\bf 36} (1961), 24-28.

\bibitem {sk} Kumari Saloni,
{\em Buchsbaumness of the associated graded ring of filtrations},
J. Algebra and its applications, {\bf 21} (2022).

\bibitem{ksmith}
K. E. Smith, {\em Tight closure of parameter ideals}, Invent. {\bf 115} (1994), 41–60.

\bibitem{vv} P. Valabrega and G. Valla, {\em Form rings and regular sequences,}
Nagoya Math. J. {\bf 72} (1978), 93-101.

\bibitem{viet} D. Q. Viet, {\em A note on the Cohen-Macaulayness of Rees algebras of filtrations,}  Comm. Algebra {\bf 21}  (1993), 221–229. 

\end{thebibliography}
\end{document}